\documentclass[12pt]{amsart}
 \usepackage[dvips]{epsfig}
 \usepackage{comment}
 \usepackage{enumerate}
 \usepackage{amsgen, amstext,amsbsy,amsopn, amsthm, amsfonts,amssymb,amscd,amsmat
 h,euscript,enumerate,url,verbatim,calc,xypic, mathtools}

 \usepackage{latexsym}
 \usepackage{graphics}
 \usepackage{color}

\newcommand{\proset}{\,\mathrel{\lower 4pt\hbox{$\scriptscriptstyle/$}
\mkern -14mu\subseteq }\,} 
 
 \newtheorem{theorem}{Theorem}[section]
  
 \newtheorem{lemma}[theorem]{Lemma}
 \newtheorem{proposition}[theorem]{Proposition}

\newtheorem{remark}[theorem]{Remark}
 
 \newtheorem{definition}[theorem]{Definition}
 
 \newtheorem{example}[theorem]{Example}

\numberwithin{equation}{section}

\usepackage{amsmath}
 \makeatother 
 
\begin{document}
 \date{\today}
 
 \title{ Lambda Module structure on higher $K$-groups}

 \author{Sourayan Banerjee and Vivek Sadhu}

 \address{Department of Mathematics, Indian Institute of Science Education and Research Bhopal, Bhopal Bypass Road, Bhauri, Bhopal-462066, Madhya Pradesh, India}
 \email{sourayan16@iiserb.ac.in, vsadhu@iiserb.ac.in}
 

\begin{abstract}
 In this article, we show that for a quasicompact scheme $X$ and $n>0,$ the $n$-th $K$-group $K_{n}(X)$ is a $\lambda$-module over a $\lambda$-ring $K_{0}(X)$ in the sense of Hesselholt.
\end{abstract}

 \maketitle

 \section{introduction}
 In \cite{Hes}, L. Hesselholt introduced the notion of module over $\lambda$-rings, i.e., $\lambda$-module. Let us first recall the definition (see Definition 2.5 and Remark 2.6 of  \cite{Hes}).
 \begin{definition}
  Let $(R, \lambda_{R})$ be a $\lambda$-ring. A $(R, \lambda_{R})$-module $(M, \lambda_{M})$ is a $R$-module $M$ and a sequence of additive maps 
$$ \lambda_{M, n}: M\to M ~~ (n\geq 1)$$ such that the following axioms hold:
\begin{enumerate}
 \item $ \lambda_{M, 1}= {\rm id}_{M}$;
 \item $ \lambda_{M, n} \lambda_{M, m}= \lambda_{M, nm}$ for all $m, n\geq 1$;
 \item $ \lambda_{M, n}(ax)=\psi^{n}(a) \lambda_{M, n}(x)$ for all $a\in R$ and $x\in M.$ Here $\psi^n$ is the $n$-th Adams operation associated to $(R, \lambda_{R}).$
\end{enumerate}
 \end{definition}
If we set $M=R$ and $ \lambda_{M, n}=\psi^{n},$ where $\psi^n$ are the Adams operations associated to $(R, \lambda_{R}),$ then $(R, \psi^{n})$ is a $(R, \lambda_{R})$-module. For a quasicompact scheme $X,$ $K_{0}(X)$ is a $\lambda$-ring with $\lambda$-operations defined by the usual exterior power on vector bundles. These exterior power operations have been extended to higher $K$-groups by several authors using homotopy theory (see \cite{Hil}, \cite{Krat}, \cite{Lev} and \cite{Soule}). Recently, a purely algebraic construction of the exterior power operations on higher $K$-groups of any quasicompact scheme is given in \cite{HKT} using Grayson's description of higher $K$-groups in terms of binary complexes. In this article, we use the exterior power operations  constructed in \cite{HKT} to give a $\lambda$-module structure on $K_{n}(X)$ over  $K_{0}(X)$ for $n>0$ and any quasicompact scheme $X.$ Here is our precise result:
\begin{theorem}\label{main result}
 For any quasicompact scheme $X$ and $n>0,$ each $K_{n}(X)$ is a $\lambda$-module over $K_{0}(X).$
\end{theorem}

 {\bf Acknowledgement:} The authors would like to thank Lars Hesselholt for his valuable comments and suggestions. The authors would also like to thank the referee for valuable comments and suggestions.
\section{Preliminaries}

\subsection{Symmetric functions} For a ring $R,$ let $R[x_1, x_2, \dots, x_n]$ denote the polynomial ring over $R$ in $n$ independent variables $x_1, x_2, \dots, x_n.$ A polynomial $f\in R[x_1, x_2, \dots, x_n]$ is said to be symmeteric function if $\pi f=f(x_{\pi(1)}, x_{\pi(2)}, \dots, x_{\pi(n)})$ for every permutation $\pi$ on $\{1, 2, \dots, n\}.$  For $1\leq k \leq n,$ the $k$-th elementary symmetric function in the variables $x_1, x_2, \dots, x_n$ is $s_k= \sum_{1\leq i_{1} <i_{2}< \dots < i_{k}\leq n} x_{i_{1}}x_{i_{2}}\dots x_{i_{k}}.$ A fundamental theorem of symmetric function says that every symmetric function $f\in R[x_1, x_2, \dots, x_n]$ can be written uniquely as a polynomial with coefficients in $R$ in the elementary symmetric functions.

\begin{example}\label{imp exm}
 The $r$-th power sum $x_{1}^{r} + x_{2}^{r} + \dots + x_{n}^{r}$ is a symmetric function, where $r>0.$ Thus, there exists a unique polynomial $Q_{r}$ in $n$ variables with integer coefficients such that 
 \begin{equation}\label{r th power}
  x_{1}^{r} + x_{2}^{r} + \dots + x_{n}^{r}=Q_{r}(s_1, s_2, \dots, s_n).
 \end{equation}

\end{example}

Let $t_{1}, t_{2}, \dots, t_{n}$ be the elementary symmetric functions for another set of variables $y_1, y_2, \dots, y_n.$ We say that a polynomial $f \in R[x_1, x_2, \dots, x_n; y_1, y_2, \dots, y_n]$ is a symmetric function if $$f(x_1, x_2, \dots, x_n; y_1, y_2, \dots, y_n)=f(x_{\pi(1)}, x_{\pi(2)}, \dots, x_{\pi(n)}; y_{\sigma(1)}, y_{\sigma(2)}, \dots, y_{\sigma(n)})$$ for every pair of permutations $\pi$ and $\sigma$ on $\{1, 2, \dots, n\}.$ Every symmetric function $f\in R[x_1, x_2, \dots, x_n; y_1, y_2, \dots, y_n]$ can be written uniquely as a polynomial with coefficients in $R$ in the elementary symmetric functions $s_1, s_2, \dots, s_n$ and $t_1, t_2, \dots, t_n.$

\subsection{The polynomials $P_{n, m}$ and $P_{n}$}\label{universal} Consider the symmetric function
$$ g(t)= \prod_{1\leq i_1< i_2< \dots < i_{m}\leq nm} (1 + x_{i_1}\dots x_{i_{m}}t)$$ of $nm$ variables. The coefficient of $t^n$ in $g(t)$ is a symmetric function and it can be expressed as a polynomial with integer coefficients in the elementary symmetric functions $s_1, s_2, \dots, s_{nm}.$ In fact, there is a universal polynomial $P_{n, m}$ with integer coefficient such that the coefficient of $t^n$ in $g(t)$ is $P_{n, m}(s_1, s_2, \dots, s_{nm}).$ Similarly, there is a universal polynomial $P_{n}$ with integer coefficient such that the coefficient of $t^n$ in $h(t)= \prod_{i, j=1}^{n} (1 + x_iy_jt)$ is $P_n(s_1, s_2, \dots, s_n; t_1, t_2, \dots, t_n).$ We can view the polynomial $P_{n}$ as the coefficient of $t^n$ in
\begin{equation}\label{obs abt Pn}
 \tilde{h}(t)= \prod_{i=1}^{n}(1+ x_{i}t_{1}t+x_{i}^{2}t_{2}t^2+\dots + x_{i}^{n}t_{n}t^{n}).
\end{equation}

 As an illustration, for $n=3,$  $P_3(s_1, s_2, s_3; t_1, t_2, t_3)$ is the coefficient of $t^3$ in
$$(1+x_1y_1t)(1+x_1y_2t)(1+x_1y_3t)
(1+x_2y_1t)(1+x_2y_2t)(1+x_2y_3t)
(1+x_3y_1t)(1+x_3y_2t)(1+x_3y_3t)$$
$$= (1+x_1t_1t+x_1^{2}t_2t^2+x_1^{3}t_3t^3)(1+x_2t_1t+x_2^{2}t_2t^2+x_2^{3}t_3t^3)(1+x_3t_1t+x_3^{2}t_2t^2+x_3^{3}t_3t^3).$$
Both these polynomials $P_{n, m}$ and $P_{n}$ will appear in the definition of $\lambda$-ring.

The reader can skip following lemma for a moment as it will be used only in Proposition \ref{psi linearity}.
\begin{lemma}\label{basic lemma linearity}
 Suppose that $t_{i}^{m}=0$ and  $t_{i}t_{j}=0$ for $m>1$ and $i\neq j,$ $1\leq i, j\leq n$ in (\ref{obs abt Pn}). Then $P_n(s_1, s_2, \dots, s_n; t_1, t_2, \dots, t_n)= Q_{n}(s_1, s_2, \dots, s_n)t_n,$ where the polynomial $Q_{n}$ as in (\ref{r th power}).
\end{lemma}
\begin{proof}
 Note that the coefficient of $t^n$ in (\ref{obs abt Pn}) is $(x_1^{n}+x_2^{n}+ \dots + x_{n}^{n})t_{n}$ because all other terms must contain $t_{i}^{m}$ or $t_{i}^{k}t_{j}^{l},$ where $m>1$ and $k, l >0.$  Hence the result by Example \ref{r th power}. 
\end{proof}

In general, it is not easy to write down the explicit formulae for $P_{n, m}$ and $P_{n}.$ However, some computations for small $m$ and $n$ has been done in \cite{Hop}. The sum of the coefficients in $P_{n, m}$ and $P_{n}$ are also calculated in \cite{Hop}. More explicitly,
\begin{lemma}\label{ sum coef of pn}
 The sum of the coefficients in $P_n$ is zero for $n>1$ and $1$ for $n=1.$
\end{lemma}
\begin{proof}
 See Theorem 2.2 of \cite{Hop}.
\end{proof}

\begin{lemma}\label{sum coef in pnm}
 The sum of the coefficients in $P_{n, m}$ is $1$ if $m$ is odd or $n=1$ and $0$ for $m$ even and $n>1.$
\end{lemma}

\begin{proof}
 See Theorem 2.3 in \cite{Hop}.
\end{proof}
 In the following lemma, we determine the coefficient of $s_{nm}$ in $P_{n,m}.$
\begin{lemma}
Let $c$ be the coefficient of $s_{nm}$ in $P_{n,m}.$ Then $c=-1$ if $n,m$ are both even, otherwise $1$. 
\end{lemma}
\begin{proof}
 See Remark \ref{coef of snm}. More specifically, see proof of Proposition \ref{psi linearity}(2) in section \ref{module structure}.
\end{proof}

\subsection{$\lambda$-rings}

 A $\lambda$-ring is a commutative unital ring $R$ together with sequence of maps $$\lambda^n: R \to R ~~ (n\geq 0),$$ called $\lambda$-operations, such that for all $x, y \in R,$ the following condition satisfy:
 \begin{enumerate}
  \item $\lambda^{0}(x)=1,$
  \item $\lambda^{1}(x)=x,$
  \item $\lambda^{n}(1)=0$ for $n\geq 2,$
  \item $\lambda^{n}(x+y)=\sum_{i+j=n} \lambda^{i}(x)\lambda^{j}(y),$
  \item $\lambda^{n}(xy)=P_{n}(\lambda^{1}(x), \dots, \lambda^{n}(x);\lambda^{1}(y), \dots, \lambda^{n}(y)),$
  \item $\lambda^{n}(\lambda^{m}(x))= P_{n, m}(\lambda^{1}(x), \dots, \lambda^{nm}(x)),$ where $P_{n}$ and $P_{n, m}$ are the universal polynomials with integer coefficients defined in subsection \ref{universal}.
 \end{enumerate}

We always mean a $\lambda$-ring $R$ as a pair $(R, \lambda_{R}:=\{\lambda^{n}_{R}\}).$

The following result follows from definition.

\begin{lemma}\label{basic observation}
 If $R$ is a $\lambda$-ring then 
 \begin{enumerate}
  \item $\lambda^n(0)=0$ for $n\geq 1;$
  \item $\lambda^{n}(-1)=(-1)^{n}$ for $n\geq 0.$
 \end{enumerate}

\end{lemma}
\subsection{Adams operations}\label{pro adams}
Given any $\lambda$-ring $R,$ one can associate sequence of functions
$$\psi^n: R \to R ~~ (n\geq 1),$$ called Adams operations, satisfying the following properties:
\begin{enumerate}
 \item each $\psi^n$ is a $\lambda$-ring homomorphism, i.e., a ring homomorphism such that $\psi^{n}\lambda^{k} = \lambda^{k}\psi^{n}$, \text{for}\; $k\geq 0.$
 \item $\psi^1={\rm id}$;
 \item for $m,n\geq 1,$ we have $\psi^n\psi^m=\psi^{mn}=\psi^m\psi^n$;
 \item for every prime number $p$ and $a\in R,$ $\psi^p(a)=a^{p}(mod~ pR).$
\end{enumerate}

 The next lemma says that Adams operations can be expressed in terms of the $\lambda$-operations. Moreover, it illustrates the uniqueness of the associated Adams operations $\psi^{n}$. 
 
 \begin{lemma}\label{adams exp as lambda}
  For a $\lambda$-ring $R,$ we have 
  $$ \psi^{n}(x)= Q_{n}(\lambda^{1}(x), \lambda^{2}(x), \dots, \lambda^{n}(x))$$ for every $x\in R$ and $n\geq 1,$ where the polynomial $Q_{n}$ was defined in (\ref{r th power}).
 \end{lemma}
\begin{proof}
 See Theorem 3.9 of \cite{yau}.
\end{proof}

We record here the Newton formula for later use.

\begin{lemma}\label{Newton}
 For any $\lambda$-ring $R,$ the relation 
 $$ \psi^{k}(x)- \lambda^{1}(x)\psi^{k-1}(x) + \dots + (-1)^{k-1}\lambda^{k-1}(x)\psi^{1}(x)= (-1)^{k+1}k\lambda^{k}(x)$$ holds for $x\in R$ and $k\geq 1.$
\end{lemma}
\begin{proof}
 See Theorem 3.10 of \cite{yau}.
\end{proof}

\section{$\lambda$-ring structure on $K_{*}(X)$}
Let $X$ be a quasicompact scheme. The tensor product induces the multiplication in the Grothendieck ring $K_{0}(X).$ The ring $K_{0}(X)$ is a $\lambda$-ring with $\lambda$-operations $$\lambda^{r}: K_{0}(X) \to K_{0}(X), ~~ (r\geq 0),$$ where the operations $\lambda^{r}$ is defined by the usual exterior power operations of vector bundles over $X.$ In \cite{HKT}, Harris, Kock and Taelman extend the exterior power operations to higher $K$-groups $K_{n}(X)$ for $n\geq 1$ using Grayson binary complex technique. We  will quickly recall the construction \cite{HKT} of the exterior power operation on higher K-groups (see below). To do this, we need the following Grayson's description of $K$-groups.

\subsection*{Grayson's $K$-groups}
 We cite \cite{Gray} and Section 1 of \cite{HKT} for a more comprehensive discussion. Let us fix some notations.
 Given an exact category $\mathcal{N},$
\begin{itemize}
 \item  $C\mathcal{N}$: The category of chain complexes in $\mathcal{N}$ that are concentrated in nonnegative degrees. In other words,  objects of $C\mathcal{N}$ are all $\mathbb{Z}_{\geq 0}$-graded objects of $\mathcal{N}$.
 
 \item $C_{b}\mathcal{N}$: The exact subcategory of $C\mathcal{N}$ of bounded chain complexes in $\mathcal{N}.$
 
 \item $C^{q}\mathcal{N}$: The full subcategories of acyclic chain complexes in $C\mathcal{N}.$
 
 \item $C_{b}^{q}\mathcal{N}$: The category of bounded acyclic chain complexes in $C\mathcal{N}.$ 
\end{itemize}

The categories $C\mathcal{N}$, $C_{b}\mathcal{N}$ and $C_{b}^{q}\mathcal{N}$ are all exact. Thus, we can iterate their construction to define $\mathbb{Z}_{\geq 0}^n$-graded objects in $\mathcal{N}$, call them $n$-dimensional multicomplexes. The category of $n$-dimensional multicomplexes in $\mathcal{N}$ is denoted by $C^n\mathcal{N}.$ Similarly, the exact categories $(C_{b})^{n} \mathcal{N},(C^{q})^{n}\mathcal{N}$ and $(C_{b}^{q})^{n}\mathcal{N}$ denote the categories of $n$-dimensional bounded, acyclic and bounded acyclic multicomplexes respectively.

A binary complex over an exact category $\mathcal{N}$ is a triple $(N_{\bullet},d,d')$, where $N_{\bullet}$ is a $\mathbb{Z}_{\geq 0}$-graded object of $\mathcal{N}$ together with two differentials $d$ and $d'$ such that $(N_{\bullet},d)$ and $(N_{\bullet},d')$ are in $C\mathcal{N}$. As before, we fix the following notations.

\begin{itemize}
 \item $B\mathcal{N}$: The category of binary chain complexes in $\mathcal{N}.$
 \item $B_{b}\mathcal{N}$: The category of bounded binary chain complexes in $\mathcal{N}.$
 
 \item $B^{q}\mathcal{N}$: The category of acyclic binary chain complexes in $\mathcal{N}.$
 
 \item $B_{b}^{q}\mathcal{N}$: The category of bounded  acyclic binary chain complexes in $\mathcal{N}.$
\end{itemize}

Note that each of these categories of binary complexes is exact. Analogously, the exact categories $(B_{b})^{n} \mathcal{N},(B^{q})^{n}\mathcal{N}$ and $(B_{b}^{q})^{n}\mathcal{N}$ denote the categories of $n$-dimensional bounded, acyclic and bounded acyclic binary  multicomplexes respectively. An $n$-dimensional binary multicomplex is a collection of $\mathbb{Z}_{\geq 0}^n$- graded objects in $\mathcal{N}$ with differentials $(d^i,\Tilde{d^i}) $ in each direction $1\leq i\leq n$ and 
the differentials satisfy the following commutativity laws whenever $i \neq j$:  
\begin{align*}
d^id^j=d^jd^i\\
\Tilde{d^i}d^j = {d^j}\Tilde{d^i}\\
d^i\Tilde{d^j} = \Tilde{d^j}d^i\\
\Tilde{d^i}\Tilde{d^j} = \Tilde{d^j}\Tilde{d^i}.
\end{align*}
We say that an $n$-dimensional binary multicomplex {\it diagonal} if the pair of differentials in some direction are equal, i.e., $d^{i}=\Tilde{d^{i}}$ for some $1\leq i\leq n.$
Since the category $(B_{b}^{q})^{n}\mathcal{N}$ is exact, we can define $K_{0}((B_{b}^{q})^{n}\mathcal{N}).$  Let $\mathcal{D}$ be a subgroup of $K_{0}((B_{b}^{q})^{n}\mathcal{N})$ generated by the classes of the diagonal bounded acyclic binary multicomplexes. A result of Grayson says that for $n\geq 1,$ $K_{n}(\mathcal{N})$ is isomorphic to $K_{0}((B_{b}^{q})^{n}\mathcal{N})/\mathcal{D}$,  which we shall use in the remainder of the article as our definition of the $K$-groups. More precisely, we have the following definition (see Corollary 7.4 of \cite{Gray} and Definition 1.3 of \cite{HKT}): 

\begin{definition}\label{HKT def}
   Let $\mathcal{N}$ be an exact category. For $n\geq 0,$ $K_{n}\mathcal{N}$ is the abelian group having generators $[N],$ one for each object $N$ of $(B_{b}^{q})^{n}\mathcal{N}$ and the relations are:
	\begin{enumerate}
		\item $[N^{'}] + [N^{''}]=[N]$ for every short exact sequence $0 \to N^{'} \to N \to N^{''}\to 0$ in $(B_{b}^{q})^{n}\mathcal{N};$
		\item $[D]=0$ if $D$ is a diagonal bounded acyclic binary multicomplex.
	\end{enumerate}
\end{definition}

The above definition differs somewhat from the Grayson's original definition. However, the Proposition 1.4 of \cite{HKT} demonstrates that working with Definition  \ref{HKT def} is not harmful. 

\subsection*{Exterior power operations on higher $K$-groups}
Let $\mathcal{P}(X)$ denote the category of vector bundles on $X.$ In \cite{HKT}, Harris, Kock and Taelman inductively construct functors
$$\varLambda_{n}^{r}: (B_{b}^{q})^{n}\mathcal{P}(X) \to (B_{b}^{q})^{n}\mathcal{P}(X) ~~ for~ all~~ r>0 ~~ and ~~ n \geq0$$ from the usual exterior power endofunctors on $\mathcal{P}(X).$ 
The idea of the construction is as follows: Start with the usual exterior power endofunctors $\varLambda^{r}$ on $\mathcal{P}(X).$ Let $\mathcal{P}(X)^{{\Delta}^{op}}$ denote the category of simplicial objects in  $\mathcal{P}(X).$ The endofunctors $\varLambda_{1}^{r}$  on the category of bounded acyclic complexes $C_{b}^{q}\mathcal{P}(X)$ are defined as $$\varLambda_{1}^{r}:= N\varLambda^{r}\Gamma: C_{b}^{q}\mathcal{P}(X) \to C_{b}^{q}\mathcal{P}(X) ~~ for~ all~~ r>0,$$ where $\Gamma: C\mathcal{P}(X) \rightarrow \mathcal{P}(X)^{{\Delta}^{op}}$,  $N: \mathcal{P}(X)^{\Delta^{op}} \rightarrow C\mathcal{P}(X)$ are given by Dold-Kan correspondence (see section 2 of \cite{HKT}). Note $N$ is inverse to $\Gamma$ upto natural isomorphism. Recursively, we can define (see Corollary 3.5 of \cite{HKT}) $$\varLambda_{n}^{r}:= N\varLambda_{n-1}^{r}\Gamma: (C_{b}^{q})^{n}\mathcal{P}(X) \to (C_{b}^{q})^{n}\mathcal{P}(X) ~~ for~ all~~ r, n >0.$$ Given an object $(P_{\bullet}, d, d^{'})$ in $B_{b}^{q}\mathcal{P}(X),$ $\varLambda_{n}^{r}(P_{\bullet})$ is independent of the differentials $d$ and $d^{'}.$ So, we can apply  $\varLambda_{n}^{r}$ individually on $(P_{\bullet}, d)\in C_{b}^{q}\mathcal{P}(X)$ and $(P_{\bullet}, d^{'})\in C_{b}^{q}\mathcal{P}(X)$ to get an object in $B_{b}^{q}\mathcal{P}(X)$ (see Lemma 4.2 of \cite{HKT}). As a result, the endofunctors   $\varLambda_{n}^{r}$  on $(C_{b}^{q})^{n}\mathcal{P}(X)$ can be lifted to endofunctors on binary multicomplexes $$\varLambda_{n}^{r}: (B_{b}^{q})^{n}\mathcal{P}(X) \to (B_{b}^{q})^{n}\mathcal{P}(X) ~~ for~ all~~ r>0 ~~ and ~~ n \geq0.$$

\begin{lemma}\label{lambda operat}
 The functors $\varLambda_{n}^{r}$ induce well-defined homomorphisms 
 \begin{equation}\label{harris ope}
  \lambda^{r}: K_{n}(X) \to K_{n}(X) ~~ {\it for} ~~ r, n>0.
 \end{equation}
\end{lemma}
\begin{proof}
 See Theorem 6.2 of \cite{HKT}.
\end{proof}
\begin{remark}\label{lambda^1}{\rm Since $\varLambda^{1}=id$ and  $N$ is inverse to $\Gamma,$
  $\varLambda_{1}^{1}=N\varLambda^{1} \Gamma$ is identity. We can observe from iteration that $\varLambda_{n}^{1}=N\varLambda_{n-1}^{1} \Gamma$ is identity because $\varLambda_{n-1}^{1}=id$ and $N$, $\Gamma$ are compositions of the functors that are inverses of each other in every given direction. For instance, $\it\Lambda^1_2:= N_hN_v\it\Lambda^1$ $\Gamma_v\Gamma_h = id$ (see Remark 3.6 of \cite{HKT}). Here the indices $h$ and $v$ represent horizontal and vertical directions, respectively. Hence, $\lambda^{1}=id.$}
\end{remark}

The graded abelian group $K_{*}(X):= \bigoplus_{n\geq 0}K_{n}(X)$ is a commutative ring with a multiplication 
$$(a_0, a_1, a_2, \dots, )\bullet (b_0, b_1, b_2, \dots )= (a_0b_0, a_0b_1 +a_1b_0, a_0b_2+a_2b_0, \dots).$$ Note that the product of any two elements in $\bigoplus_{n\geq 1}K_{n}(X)$ is zero. Each $K_{n}(X)$ is a $K_{0}(X)$-module via $[P].[Q_{\bullet}]:=[P\otimes Q_{\bullet}],$ where $P$ in $\mathcal{P}(X)$ and $Q_{\bullet}$ in $(B_{b}^{q})^{n}\mathcal{P}(X).$ Furthermore, the exterior power operations $\lambda^{r}: K_{*}(X) \to K_{*}(X), r\geq 0$ defined by the formula
\begin{equation}\label{oper on K_*}
 \lambda^{r}((a_0, a_1, a_2, \dots))=(\lambda^{r}(a_0), \sum_{i=0}^{r-1}\lambda^{i}(a_0)\lambda^{r-i}(a_1), \sum_{i=0}^{r-1}\lambda^{i}(a_0)\lambda^{r-i}(a_2), \dots).
\end{equation}
\begin{lemma}\label{K(X) is Lambda}
 The operations $\lambda^{r}$ defined in (\ref{oper on K_*}) satisfy the axioms of $\lambda$-ring. In otherwords, $K_{*}(X)$ is a $\lambda$-ring.
\end{lemma}
\begin{proof}
 See Theorems 7.1 and 8.18 of \cite{HKT}.
\end{proof}

\section{ $\lambda$-module structure on higher $K$-groups}\label{module structure}
\begin{proposition}\label{psi linearity}
 Let $a\in K_{0}(X)$ and $x\in K_{n}(X)$ for $n>0.$ The group homomorphisms $\lambda^{r}: K_{n}(X) \to K_{n}(X) ~~ {\it for} ~~ r, n>0 $ defined in (\ref{harris ope}) satisfy the following:
 \begin{enumerate}
  \item $\lambda^{r}(a.x)=\psi^{r}(a)\lambda^{r}(x).$ Here $\psi^{r}$ is the $r$-th Adams operation on $K_{0}(X).$
  
  \item $\lambda^{r}\lambda^{s}=c\lambda^{rs},$ where $c$ is $-1$ for $r,s$ both even and $1$ otherwise.

 \end{enumerate}

\end{proposition}

\begin{proof}
 (1) Let $\underline{a}=(a, 0, \dots )\in K_{*}(X)$ and $\underline{x}=(0, 0, \dots, x, 0, \dots )\in K_{*}(X).$ Since $K_{*}(X)$ is a $\lambda$-ring (see Lemma \ref{K(X) is Lambda}), $$\lambda^{r}(\underline{a}\bullet\underline{x})= P_{r}(\lambda^{1}(\underline{a}), \dots, \lambda^{r}(\underline{a});\lambda^{1}(\underline{x}), \dots, \lambda^{r}(\underline{x})).$$ Write $\underline{ax}$ for $(0, 0, \dots, ax, 0, \dots )$, $\underline{\lambda^{i}(a)}$ for $(\lambda^{i}(a), 0, \dots)$ and $\underline{\lambda^{i}(x)}$ for $(0, 0, \dots, \lambda^{i}(x), 0, \dots).$ Using the multiplication rule on $K_{*}(X),$ we have 
 
 \begin{align*}
  (0, \dots, \lambda^{r}(ax), 0, \dots)=\lambda^{r}(\underline{ax})&= P_{r}(\underline{\lambda^{1}(a)}, \dots, \underline{\lambda^{r}(a)};\underline{\lambda^{1}(x)}, \dots, \underline{\lambda^{r}(x)})\\
  &= Q_{r}(\underline{\lambda^{1}(a)}, \dots, \underline{\lambda^{r}(a)})\underline{\lambda^{r}(x)}~~ (by~ Lemma~ \ref{basic lemma linearity})\\
  &=(Q_{r}(\lambda^{1}(a), \dots, \lambda^{r}(a)), 0, \dots)\underline{\lambda^{r}(x)}\\
  &=(0, \dots, Q_{r}(\lambda^{1}(a), \dots, \lambda^{r}(a))\lambda^{r}(x), 0, \dots).
 \end{align*} Therefore, $\lambda^{r}(ax)=Q_{r}(\lambda^{1}(a), \dots, \lambda^{r}(a))\lambda^{r}(x).$ We get the desired assertion by Lemma \ref{adams exp as lambda}.

 (2) Note that $\lambda^{r}\lambda^{s}(\underline{x})=P_{r, s}(\lambda^{1}(\underline{x}), \dots, \lambda^{rs}(\underline{x}))=c\lambda^{rs}(\underline{x})$ because all products appearing in the polynomial $P_{r,s}$ are trivial. By applying Newton formula (see Lemma \ref{Newton}) for the $\lambda$-ring $K_{*}(X)$ and Adams operations $\psi^{*},$ we get $\psi^{rs}(\underline{x})=(-1)^{rs+1}rs\lambda^{rs}(\underline{x}).$ Since each $\psi^{r}$ is a ring homomorphism, we also have $\psi^{r}\psi^{s}(\underline{x})=(-1)^{r+1}(-1)^{s+1}rs\lambda^{r}\lambda^{s}(\underline{x}).$ We know $\psi^{rs}=\psi^{r}\psi^{s}.$ Thus, $(-1)^{rs+1}=(-1)^{r+1}(-1)^{s+1}c$ because the characteristic of $K_{*}(X)$ is zero (see Proposition 1.29 of \cite{yau}). Hence the result. \end{proof}
 
 \begin{remark}\label{coef of snm}
 \rm {The proof of Proposition \ref{psi linearity}(2) basically determines the coefficient of $s_{nm}$ in the polynomial $P_{n,m}(s_1, s_2, \dots, s_{nm}).$}
 \end{remark}

 {\it Proof of Theorem \ref{main result}:} For $r, n>0,$ we define $\lambda_{K_{n}(X), r}:= (-1)^{r-1}\lambda^{r}: K_{n}(X) \to K_{n}(X),$ where $\lambda^{r}$ as in (\ref{lambda operat}). The result now follows from Remark \ref{lambda^1} and Proposition \ref{psi linearity}.\qed


\begin{thebibliography}{AAA}
 
  
  \bibitem{Gray} D. R. Grayson, {\it Algebraic $K$-theory via binary complexes}, Journal of American Mathematical Society, Vol {\bf 25}, Number {\bf 4} (2012), 1149-1167.
 
  
  \bibitem{HKT} T. Harris, B. Kock and L. Taelman, {\it Exterior power operations on higher $K$-groups via binary complexes}, Annals of $K$-theory, vol.2 ({\bf 3}), (2017) 409-449.
  
  \bibitem{Hil} H. L. Hiller, {\it $\lambda$-rings and algebraic $K$-theory}, J. Pure and Applied Algebra, {\bf 20}(3) (1981), 241-266.
  
  
  \bibitem{Hes} L. Hesselholt, {\it The big de Rham-Witt complex}, Acta. Math., {\bf 214} (2015), 135-207.
  
  \bibitem{Hop} John R. Hopkinson, {\it Universal polynomials in lambda rings and the $K$-theory of the infinite loop space tmf}, Thesis (Ph.D.)-Massachusetts Institute of Technology. 2006.
  
  \bibitem{Krat} C. Kratzer, {\it $\lambda$-structure en $K$-theorie algebrique}, Comment. Math. Helv. {\bf 55}(2) (1980), 233-254.
  
 
  
  \bibitem{Lev} M. Levine, {\it Lambda-operations, $K$-theory and motivic cohomology}, pp. 131-184 in Algebraic $K$-theory (Toronto, ON, 1996) edited by V. P. Snaith, Field Inst. Commun. 16, Amer. Math. Soc., Providence, RI, 1997.
  
  \bibitem{Soule} C. Soule, {\it Operations en $K$-theorie algebrique}, Canadian J. Math {\bf 37}(3) (1985), 488-550.
  
  
  
 
  
  
  
  \bibitem{yau} D. Yau, {\it Lambda-rings}, World Scientific, Hackensack, NJ, 2010.
  
 
 \end{thebibliography}
\end{document}